\documentclass{amsart}
\usepackage{amsmath, amsthm, amssymb}
\newtheorem{thrm}{Theorem}
\usepackage{graphicx}
\usepackage{placeins}
\usepackage{float}
\usepackage{subfigure}
\usepackage{xcolor}
\begin{document}
\title[Implicit method using nodal radial basis functions]{A fully implicit method using nodal radial basis functions to solve the linear advection equation}
\thanks{One of us (PAG) would like to thank Carter Ball and Carlos Cabrera with setting up some simulations used in this paper. }
\author[P.-A. Gourdain, M. B. Adams, M. Evans, H. R. Hasson, et al.]{P.-A. Gourdain, M. B. Adams, M. Evans, H. R. Hasson, J. R. Young, I. West-Abdallah}\address{Physics and Astronomy Department, University of Rochester, Rochester, New York, 14627, USA}
\date{23/03/2022}
\begin{abstract} Radial basis functions are typically used when discretization sche-mes require inhomogeneous node distributions. While spawning from a desire to interpolate functions on a random set of nodes, they have found successful applications in solving many types of differential equations. However, the weights of the interpolated solution, used in the linear superposition of basis functions to interpolate the solution, and the actual value of the solution are completely different. In fact, these weights mix the value of the solution with the geometrical location of the nodes used to discretize the equation. In this paper, we used nodal radial basis functions, which are interpolants of the impulse function at each node inside the domain. This transformation allows to solve  a linear hyperbolic partial differential equation using series expansion rather than the explicit computation of a matrix inverse. This transformation effectively yields an implicit solver which only requires the multiplication of vectors with matrices. Because the solver requires neither matrix inverse nor matrix-matrix products, this approach is numerically more stable and  reduces the error by at least two orders of magnitude, compared to other solvers using radial basis functions directly.  Further, boundary conditions are integrated directly inside the solver, at no extra cost.  The method is naturally conservative, keeping the error virtually constant throughout the computation. 
 \end{abstract}

\maketitle
\section{Introduction}

Radial basis function  interpolation  is one of the few methods that can  approximate 
across a \emph{d}-dimensional space a function only defined on a
randomly distributed set of \emph{n} nodes
$x_{1},x_{2},...,x_{n}\ \in \mathbb{R}^{d}$ \cite{hardy,franke}.
While initially used for
interpolation problems, this method can be used to defined surfaces in
multiple dimensions \cite{carr} or solve partial
differential
equations \cite{kansa90a,kansa90b,fass99,fass09}.
One important characteristic of radial basis functions is their
definition using the relative positions of nodes, obtained from the
Euclidian norm \textbar\textbar.\textbar\textbar{}\textsubscript{d},
rather than their absolute location in space. For
$x \in \mathbb{R}^{d}$ we define the radial basis function (RBF) at
every node $x_{j}$ as $\Phi_{x_{j}}(x) = \phi(||x - x_{j}||_{d})$,
which we will also write as $\Phi(x - x_{j})$. Typically, $\Phi$ is
normalized, i.e. $\Phi(0) = 1$. New functions can be generated by
scaling of the modal function $\phi$ by a factor $\alpha$, giving
the standard definition of the radial basis function
$\Phi_{\alpha,x_{j}}$ as
\begin{equation*}\label{eq:1}
\def\labelenumi{(\arabic{enumi})}
\Phi_{\alpha,x_{j}}(x) = \phi(||x - x_{j}||_{d}/\alpha),
\end{equation*}
which we will also write as $\Phi_{\alpha}(x - x_{j})$. We use the
width parameter $\alpha$ rather than the usual shape factor (i.e.
1/$\alpha$) in this paper because we will compare the radial basis
function spread to the domain size throughout this paper.

A continuous function \emph{f} can be approximated on a finite set of
nodes $U = \{ x_{1},x_{2},...,x_{n}\}$ using radial basis functions by
computing a set of weights $\omega_{j}$ defined by
\begin{equation*}
\forall x_{i} \in U,f\left( x_{i} \right) = \sum_{j}^{}\omega_{j}\Phi_{\alpha}\left( x_{i} - x_{j} \right).
\end{equation*}
The weights $\omega_{j}$ can be found by solving the linear system

\[\begin{bmatrix}
\Phi_{\alpha}(x_{1} - x_{1}) & & \text{...} & & \Phi_{\alpha}(x_{1} - x_{n}) \\
\text{...} & & \text{...} & & \text{...} \\
\Phi_{\alpha}(x_{n} - x_{1}) & & \text{...} & & \Phi_{\alpha}(x_{n} - x_{n}) \\
\end{bmatrix}\begin{bmatrix}
\omega_{1} \\
\text{...} \\
\omega_{n} \\
\end{bmatrix} = \begin{bmatrix}
f(x_{1}) \\
\text{...} \\
f(x_{n}) \\
\end{bmatrix}\]

written in compact form as
$\lbrack\Phi_{\alpha}\rbrack\lbrack\omega\rbrack = \lbrack f\rbrack.$
To solve this system, we need to find the inverse of the matrix
$\lbrack\Phi_{\alpha}\rbrack$ and compute the weights $\omega_{j}$
using
\begin{equation}\label{eq:3}
\lbrack\omega\rbrack = \lbrack\Phi_{\alpha}\rbrack^{- 1}\lbrack f\rbrack.
\end{equation}
The radial basis function $\Phi$ is said to be definite positive when
$\lbrack\Phi\rbrack$ is invertible, supposing \emph{U} does not have any redundant nodes (i.e. \emph{x\textsubscript{i}}=\emph{x\textsubscript{j}} while $i\neq j$). Once the weights are known, the function $f$ can be interpolated between nodes using the function $\bar{\overline{f}}$
defined by
\begin{equation}\label{eq:4}
\overline{f}(x) = \sum_{j}^{}\omega_{j}\Phi_{\alpha}\left( x - x_{j} \right).
\end{equation}
 
Since radial basis functions can interpolate any smooth function using a linear combination of differentiable functions, it quickly spawned differential equation solvers for
elliptic \cite{liu}, hyperbolic \cite{xliu}, parabolic \cite{zamo} or shallow-water \cite{flyer09} equations using different approaches such as spectral
\cite{shiva} or backward substitution \cite{reut14,zhang20} methods. Even differential equations with fractional operators \cite{mirz,lin22}, curvilinear coordinates \cite{Wright10} or complex boundary conditions \cite{karageo21} can be solved using this technique. The ease in defining spacial and temporal derivatives is probably the main reason this method has found universal applications. 

However, one major issue raised by Eq. (\ref{eq:4}) is evident.  The interpolated function $\overline{f}$ is now defined
in term of the weights $\omega_{j}$. This becomes an issue when
solving differential equations using radial basis functions. For
instance, the value of the function might be required to compute the
value of another function or match a set of boundary conditions. Further
if we want to interpolate a new function g, then all the weights
$\omega_{j}$ must be computed again.

In the rest of the paper, we first define a set of nodal radial basis functions (NRBF) that interpolates the impulse function. These functions form an orthonormal basis for the inner product of interpolated function on \emph{U}. We then summarize the basic properties of
NRBF formed using RBF with compact support, then we present the theory
behind our linear advection equation solver and compare it to standard
solvers. Finally, we conclude by showing how NRBF can be trivially
extended to solve  the advection equation with a velocity which varies across the domain . It is important to
note that the method is completely independent of the number of spatial
dimensions by construction. As a result, we will not look at
multidimensional cases in this paper. While we do not claim that the
method will perform well in a larger number of dimensions, the solver
proposed is clearly dimension agnostic.

\section{Definition of nodal radial basis functions\label{definition}}
The solution to avoid this problem is relatively straightforward. Rather than using radial basis functions directly, which have well defined, yet poorly matched, values at the node points, we can used them to interpolate the impulse functions $\delta(x-x_i)$ first. Once these new functions are defined, interpolation is trivial since the weights for each interpolant $\Psi_{x_i}$ is $f(x_i)$. To construct them, we can rewrite Eq. (\ref{eq:3}) as
\begin{equation}\label{eq:5}
\lbrack\omega\rbrack = \lbrack\Phi_{\alpha}\rbrack^{- 1}\sum_{j}^{}f(x_{j})\lbrack\delta_{j}\rbrack.
\end{equation}
Here the vectors
$\left\lbrack \delta_{j} \right\rbrack = \lbrack\delta(x_{i} - x_{j})\rbrack^{T}$
are defined using the impulse function
$$\forall x,y \in \mathbb{R,}\,\delta(x - y) = \left\{ \begin{matrix}
1\ \text{if}\ x = y \\
0\ \text{if}\ x \neq y \\
\end{matrix} \right.\ $$
This decomposition allows to create a series of interpolants on \emph{U}
for each translated impulse function $\delta(x - x_{j})$
\begin{equation}\label{eq:6}
\forall x_{i},x_{j} \in U,\,\Psi_{\alpha,x_{j}}\left( x_{i} \right) = \delta\left( x_{i} - x_{j} \right).
\end{equation}
They can be expressed in term of our radial basis functions as
\begin{equation}\label{eq:7}
\Psi_{{\alpha,x}_{j}}(x) = \sum_{i}^{}\Omega_{\text{ij}}\Phi_{\alpha}(x_{i} - x).
\end{equation}
and their weights $\Omega_{\text{ij}}$ can be computed given in Eq. (\ref{eq:3})
\begin{equation*}
\lbrack\Omega_{j}\rbrack = \lbrack\Phi_{\alpha}\rbrack^{- 1}\lbrack\delta_{j}\rbrack
\end{equation*}
It is interesting to note that these weights simply are the elements of
the matrix $\lbrack\Phi_{\alpha}\rbrack^{- 1}$.

\begin{thrm}
 \emph{The interpolant of $f$ formed using $\Phi_{x_i}$, i.e. $\overline{f}(x)= \sum_{j}^{}\omega_{j}\Phi_{\alpha}\left( x_{i} - x_{j} \right)$, and the interpolant of $f$ formed using $\Psi_{x_i}$, i.e. $\overline{\overline{f}}(x)=\sum_{i}f(x_{i})\Psi_{\alpha,x_{i}}(x)$, are identical.}\\
 \end{thrm}

\begin{proof}
We can write
$\omega_{j} = \sum_{i}^{}f(x_{i})\Omega_{\text{ij}}$ using Eq. (\ref{eq:5}) and
Eq. (\ref{eq:4}) becomes
$$\bar{f}(x) = \sum_{j}^{}{\sum_{i}^{}{f(x_{i})\Omega_{\text{ij}}\Phi_{\alpha}(x - x_{j})}},$$
The sum operators can be easily permuted to give
$$\bar{f}(x) = \sum_{i}^{}f(x_{i})\sum_{j}^{}\Omega_{\text{ij}}\Phi_{\alpha}(x - x_{j}).$$
Since $\Omega_{\text{ij}} = \Omega_{\text{ji}}$ and
$\Phi_{\alpha}(x - x_{j}) = \Phi_{\alpha}(x_{j} - x)$ since
$|\left| x - x_{j} \right||_{d} = ||x_{j} - x||_{d}$, we can rewrite the above equation as
\begin{equation}\label{eq:9}
{\overline{f}}(x) = \sum_{i}f(x_{i})\Psi_{\alpha,x_{i}}(x).
\end{equation}
So $\overline{\overline{f}}\equiv\overline{f}$.
\end{proof}
While the functions $\Psi_{\alpha,x_{i}}$ are constructed using radial
basis functions, they are fundamentally different. Figure \ref{fig:1} shows the
$\Psi_{\alpha,0}$ together with one of the radial basis functions used
for its construction. In fact, it looks more similar to  the barycentric form of rational interpolants  \cite{float}.
\begin{figure}[ht]
\includegraphics[width=3in]{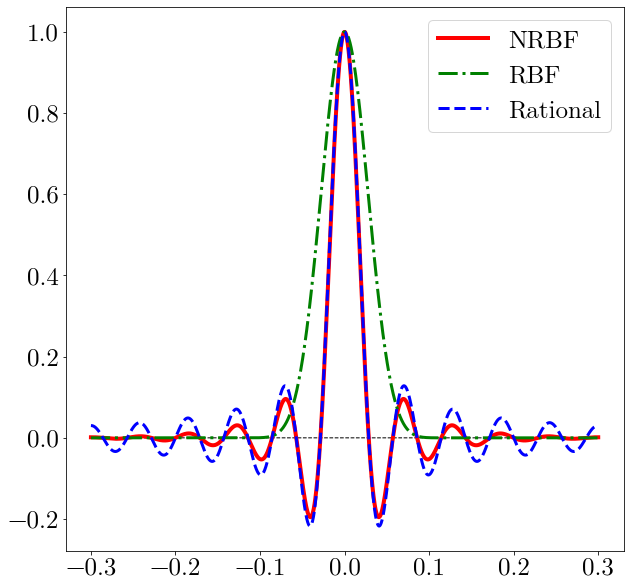}
\caption{The nodal radial basis function (NRBF) with 0 as its main
support node and the radial basis function (RBF) that was used to
generate it. The rational basis function centered on the same node is
also given. We show only a portion of the interval {[}-1,1{]} for
clarity. Each interval nodes (besides 0) can be found where the nodal
radial basis function crosses the x-axis.}
\label{fig:1}
\end{figure}
Clearly, Eq. (\ref{eq:4}) and Eq. (\ref{eq:9}) are equivalent and the interpolants
$\overline{f}$ and ${\overline{\overline{f}}}$ are
mathematically identical. While Eq. (\ref{eq:7}) is used in the radial point
interpolation collocation method (RPICM) \cite{wang}  or pseudo-spectral methods \cite{fasshauer2005rbf,fasshauer2007meshfree}, it is important to note that we defined here the function $\Psi_{\alpha,x_{j}}$ to be the interpolants of
the translated impulse function $\delta(x - x_{j})$, a definition similar to the construction of cardinal functions\cite{buhmann1990multivariate}. However, there is no advantage in using
 $\Psi_{\alpha,x_{j}}$ as an interpolation method. The equivalence
between Eq. (\ref{eq:4}) and Eq. (\ref{eq:9}) shows that we get the exact same interpolant, yet we need to compute many radial basis functions at every node in Eq. (\ref{eq:9}) and then perform 
a linear system solve to form a single $\Psi_{\alpha,x_{j}}$. 

However, the functions $\Psi_{\alpha,x_{j}}$ can be used efficiently in
solving differential equations, since the interpolating weights of a
function \emph{f} using $\Psi_{\alpha,x_{j}}$ are the values of the
functions \emph{f} at the nodes \emph{x\textsubscript{j}}. The main
reason is that the functions $\Psi_{\alpha,x_{j}}$ are an orthonormal
basis of $\overline{U}$, the space of interpolants on \emph{U}, for
the inner product defined as
\begin{equation*}
<\overline{f}.\overline{g}> = \sum_{k = 1}^{n}|{\overline{f}\left( x_k \right)\overline{g}(x_k)|},
\end{equation*}
for any two interpolants $\overline{f}$ and $\overline{g}$ in
$\overline{U}$.\\*
\begin{thrm} \emph{The set of functions $\{\Psi_{x_j}\}_{x_j\in U}$ forms an orthonormal basis for the inner product of interpolated functions } $<.>$.\\
\end{thrm}
\begin{proof} Using the definition of the nodal radial basis function given by Eq. (\ref{eq:6}) $$\forall i,j\,\,<\Psi_{x_i}.\Psi_{x_j}>=\sum_{k=1}^{n}|\Psi{x_i}(x_k)\Psi_{x_j}(x_k)|=\sum_{k=1}^{n}\delta_{ik}\delta_{jk}$$
So, $<\Psi_{x_i}.\Psi_{x_j}>=0\iff i\ne j$ and $<\Psi_{x_i}.\Psi_{x_j}>=1\iff i= j$.
\end{proof}
Based on the definition of the nodal
radial function from Eq. (\ref{eq:7}), the spatial derivative
$\partial_{k}\Psi_{\alpha,x_{j}}$ can be computed easily

\begin{equation}\label{eq:11}
\partial_{k}\Psi_{\alpha,x_{j}}(x) = \sum_{i}^{}\Omega_{\text{ij}}\partial_{k}\Phi_{\alpha}(x_{i} - x).
\end{equation}
Reverting to the modal basis function at this stage, we now rewrite Eq.
(11) as
\begin{equation}\label{eq:12}
\partial_{k}\Psi_{\alpha,x_{j}}(x) = \sum_{i}^{}{\Omega_{\text{ij}}/\alpha}(d_{r}\phi)(||x_{i} - x||_{d}/\alpha)\partial_{k}||x_{i} - x||_{d}.
\end{equation}
 While Eq. (\ref{eq:12}) uses the elements of
$\lbrack\Phi_{\alpha}\rbrack^{- 1}$ we will show later that it is not necessary to compute the matrix inverse of $[\Phi_{\alpha}]$. Rather, the solver uses a Cholesky decomposition to compute the relevant spatial derivatives . The functions $\partial_{k}\Psi_{\alpha,x_{j}}(x)$
only depend on the node distribution and should be recomputed only when
nodes change locations, or when nodes are added (or dropped).

$\Psi_{\alpha,x_{j}}(x)$ is nodal in the sense that it is defined
using nodes, rather modal, i.e. the scaled translation of the modal
function $\phi$ used in Eq. (\ref{eq:1}). It is also radial in the sense that it solely defined by a series of Euclidian distances . Since the family
of functions defined by all the points inside the domain form an
orthogonal basis for this domain, we call the functions
$\Psi_{\alpha,x_{j}}$ nodal radial basis functions (NRBF) in the
remainder of this paper.

Another key advantage of nodal radial basis functions, over more conventional discretization schemes, boils down to computing discretized derivatives as a sum of exact derivatives, given by Eq. (\ref{eq:12}). Exact derivatives, as opposed to discretized derivatives, are always conservative i.e. the derivative operator $\partial_{q}$ with respect to the variable $q$ is such that for any smooth function $g$
\begin{equation}\label{eq:13}
\int_{\Omega}^{}{\partial_{q}g\text{dv}} = \int_{\partial\Omega}^{}{gn_{q}\text{ds}}.
\end{equation}
This is the divergence theorem applied to a single direction.\\*
\begin{thrm} \emph{The derivative operator ${\widetilde{\partial}}_{q}$, used in Eq. (\ref{eq:21}) and defined as}
\begin{equation*}
{\widetilde{\partial}}_{q}f = \sum_{i}^{}\overline{f}(x_{i})\partial_{q}\Psi_{\alpha,x_{i}}
\end{equation*}
\emph{for any interpolant $\overline{f}$ is conservative.}\\
\end{thrm}
\begin{proof} The derivative can be written as
\begin{equation*}
{\widetilde{\partial}}_{q}f = \sum_{i}^{}f(x_{i})\sum_{j}^{}\Omega_{\text{ij}}\partial_{q}\Phi_{\alpha}(x - x_{j})
\end{equation*}
using Eq. (\ref{eq:11}). Since ${\widetilde{\partial}}_{q}$ is a finite sum of
exact derivatives, it trivially verifies Eq. (\ref{eq:13}). 
\end{proof}

\section{Basis characteristics of nodal radial basis functions using
compact radial basis functions}

In this section, we focus on the one-dimensional case, mostly with
homogeneously distributed nodes, to illustrate the properties of nodal
radial basis functions in a simple framework. However, this work can be
directly extended to multiple space dimensions and random node
distributions.

The key advantage of nodal radial basis functions is
$\forall x_{i},x_{j}\,\psi_{\alpha,x_{j}}(x_{i}) = \delta_{\text{ij}}$
by construction, a property shared with Lagrange polynomials and
rational interpolants\textsuperscript{8}. Yet, they retain the
multivariate interpolation capabilities on a set of randomly distributed
nodes, a theme central to radial basis function interpolation. Unless
otherwise stated, the modal radial basis functions used to form
$\psi_{\alpha,x_{j}}$ are compact Wendland functions \cite{wend95} defined by
$$\psi_{p,0}(r) = (1 - r)_{+}^{p} = \left\{ \begin{matrix}
(1 - r)^{p} \\
0 \\
\end{matrix}\ \begin{matrix}
\text{\ \ \ \ \ for\ }0 \leq r \leq 1 \\
\text{for\ }r > 1 \\
\end{matrix} \right.\ $$ and
$$\psi_{p,q}(r) = \mathfrak{I}^{q}\psi_{p,0}\text{\ \ for\ }0 \leq r \leq 1,$$
where \emph{p} and \emph{q} are integers. The operator $\mathfrak{I}$
above is defined as $\mathfrak{I}f(r) = \int_{r}^{\infty}{f(t)tdt}$ for $0 \leq r$.
Wendland functions are $C^{k}$ and can be computed analytically and
lead to a strictly positive definite matrix
$\lbrack\Phi_{\alpha}\rbrack$ in $\mathbb{R}^{d}$, where
\emph{d}\textless{}\emph{p} and \emph{k}=2\emph{q}.

Figure \ref{fig:1} shows the difference between the nodal radial basis function
$\psi_{\alpha = 7\varepsilon_{0},x = 0}$ and the modal radial
basis function $\phi_{\alpha = 7\varepsilon_{0}}$ that was used to
generate it on the interval $[-1,1]$, with $\varepsilon_{0}$ the
distance between two consecutive nodes. A rational basis
function \cite{float}
is shown for comparison. While the nodal and rational basis functions
are similar close to the main support node, the radial basis functions
quickly drop to zero away from this node. To this extent, the nodal
radial basis functions resemble more compact functions like
interpolets \cite{desl}

However, some radial basis functions are not ideal candidates to build
nodal radial basis functions. For instance, Gaussian functions, that are
infinitely smooth can trigger Runge's phenomenon \cite{forn} in the interpolation
of the impulse function, leading to the construction of our nodal radial
basis function, for width parameters relatively small, as shown in
Figure \ref{fig:2}. There is no Runge's phenomenon for Wendland functions, even
for large width parameters. Since this work focuses on solving a partial
differential equation, getting large oscillations near the boundary is
problematic for two reasons. First, large oscillations are usually
caused by ill-conditioned matrices, a well-known problem when working
with radial basis functions, that will limit the precision of the
interpolation or the differential equation solver. Second, these solvers
are sensitive to boundary condition errors and, using functions that are
widely oscillatory there would clearly be problematic.
\begin{figure}[ht]
\includegraphics[width=2.7in]{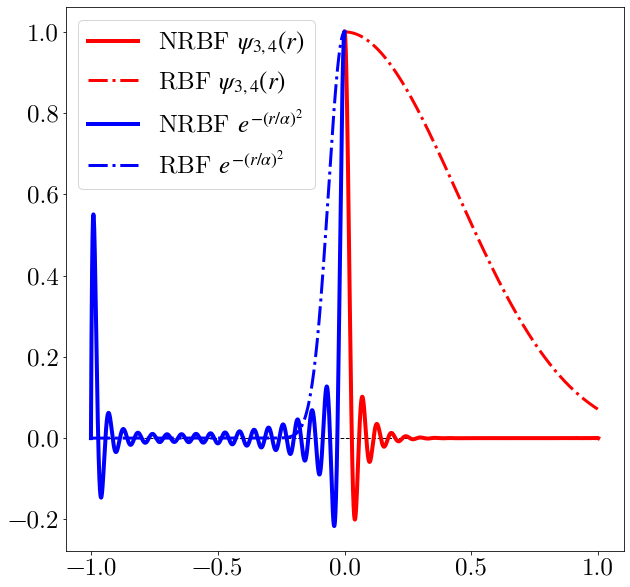}
\caption{Left: The Gaussian nodal radial basis function is wide enough to trigger Runge's phenomenon in the nodal radial basis function at the edge of the domain. Right: Wendland functions (here $\psi_{3,4}$) yield nodal radial basis functions with no boundary instabilities.}
\label{fig:2}
\end{figure}
The number of sideband oscillations surrounding the main support node
can be controlled by the degree of smoothness and the width parameters
of the radial basis function. Figure \ref{fig:3}-a shows that an increase in the
smoothness of radial basis functions (using Wendland functions
$\psi_{3,1}$ through$\psi_{3,4}$ leads
to larger sideband oscillations, a direct consequence of a weaker
exponential decay of scaling coefficients. However, this trend reverses
if the width parameter is too small (as seen in Figure \ref{fig:3}-b).
\begin{figure}[ht]
a)\includegraphics[width=3in]{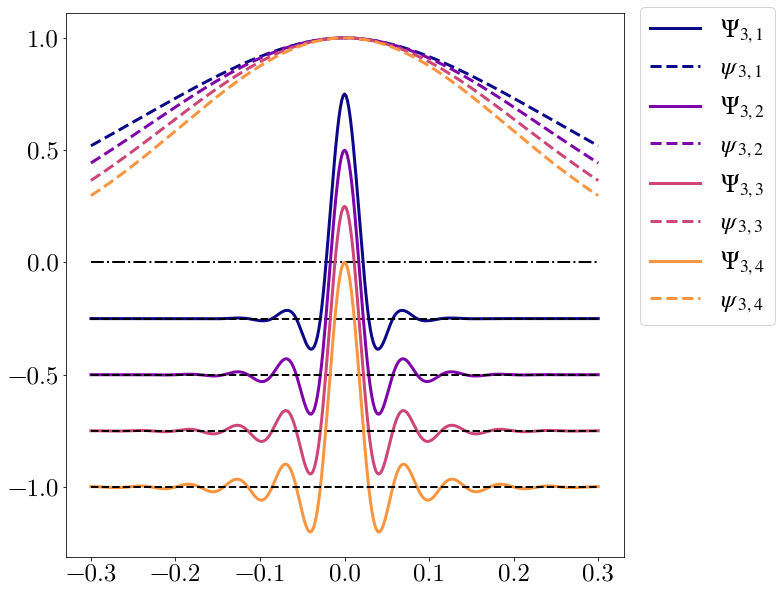}
b)\includegraphics[width=3in]{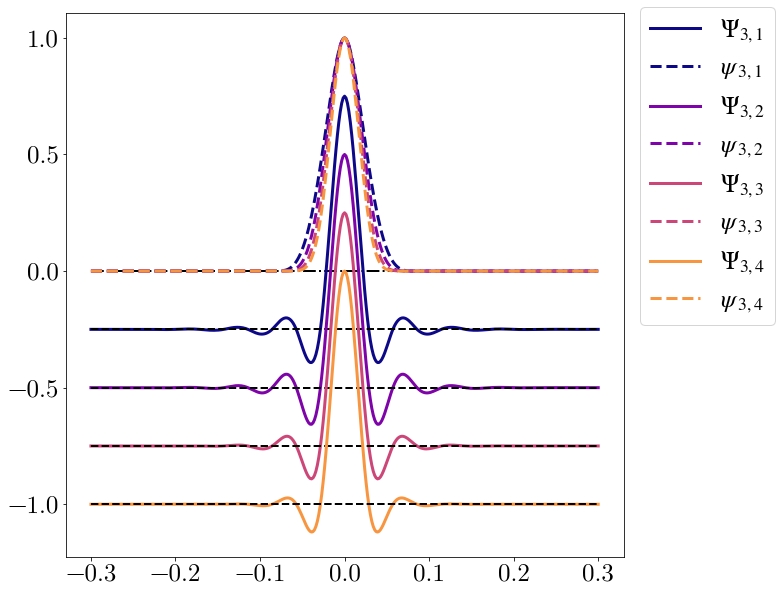}
\caption{
Nodal radial basis functions $\psi_{d,k}$ for Wendland
functions $\psi_{d,k}$ of different degrees of smoothness k for
width parameters a) $\alpha=35\varepsilon_0$ and b) $\alpha=3\varepsilon_0$.
Here $\varepsilon_0$ is the smallest distance between neighboring
nodes in U. The interval was truncated for clarity. The nodal functions
are plotted using a solid line. A vertical offset was added to
disentangle the oscillations of the function. The corresponding Wendland
functions are plotted using interrupted lines and were not given any
offset. Note here the nodal function subscripts identify the Wendland
functions rather than the geometrical scaling and translating parameters
as it is done in the text.}
\label{fig:3}
\end{figure}

This observation leads to the question of compactness. A priori, nodal
radial basis functions are not compact. Even if
$\lbrack\Phi_{\alpha}\rbrack$ is a banded matrix (i.e. non-zero
elements are close to the diagonal),
$\lbrack\Phi_{\alpha}\rbrack^{- 1}$ not necessarily banded and, in
fact, can be dense. Figure \ref{fig:4} shows the scaling coefficients for several
nodal radial basis functions with the same support node (0 in this case
case). The logarithmic scale clearly shows the exponential decay of the
coefficient away from the main support node. In some cases, the
exponential decay is not constant, and it depends on the smoothness and
width parameter. In fact, after the initial decay, the function
rebounds. This rebound gets pushed further out as the width parameter
increases (see Figure \ref{fig:4}-a and b). As this point, boundary effects become
dominant, leading to a weaker exponential decay. We notice that these
trends tend to disappear as the function smoothness increases (see
Figure \ref{fig:4}-d). Yet, it is relatively easy to truncate the functions
generated by radial basis functions with low smoothness. Truncation is
simply enforced by dropping all scaling coefficients that are o(1). As
the smoothness and width parameter increase, more coefficients should be
retained. On small domains, like the one used in this section,
truncation is not possible since there is no coefficient o(1) for
$\psi_{3,4}$ when width parameters are larger than 15
nodes, as shown in Figure \ref{fig:4}-d.
\begin{figure}[ht]
\includegraphics[width=3.25in]{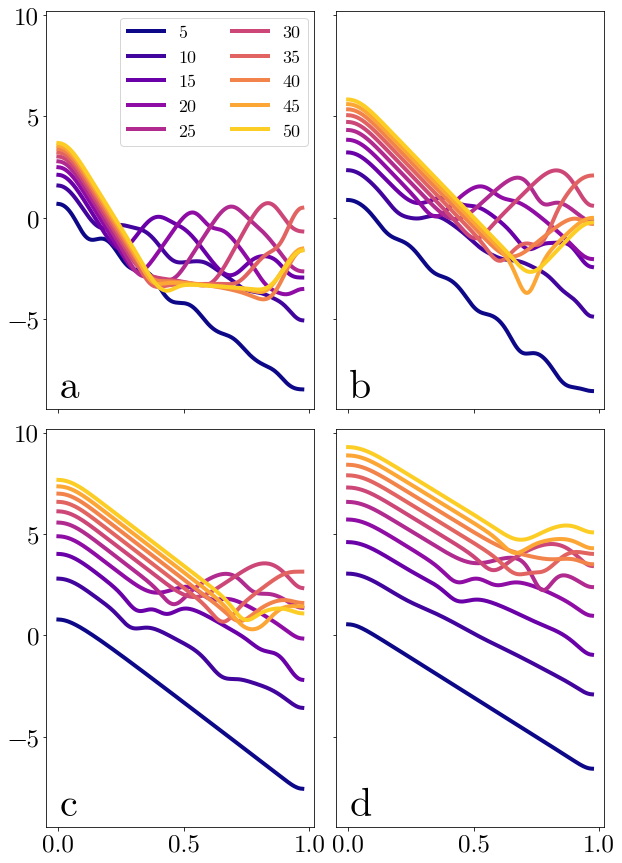}
\caption{The log\textsubscript{10} value of the scaling coefficients
used to form nodal radial basis function centered on 0 and using a)
$\psi_{3,1}$, b)$\psi_{3,2}$, c)$\psi_{3,3}$
and d) $\psi_{3,4}$ for different width parameters, given in
number of nodes. Each coefficient is associated with a node location.}
\label{fig:4}
\end{figure}

Figure \ref{fig:3} also shows that that, after some initial variations, the shape
of the nodal radial basis functions remains virtually the same as the
width parameter increases. This is a clear departure from radial basis
functions, where the width parameter clearly impacts the shape of the
function. As the width parameter grows (and the shape parameters goes to
0), radial basis functions become remarkably flat, leading to better
interpolants at the cost of ill-conditioned linear systems \cite{schab}. The resulting
trade-off between interpolation accuracy and numerical instabilities is
difficult to quantify, even if qualitative arguments can be inferred
from a wide range of studies (e.g. Ref. \cite{fass07}). Nodal radial basis
functions are less dependent on the width parameter than radial basis
functions. Figure \ref{fig:5}-a to c shows that when changing the width parameter
by 10\%, the maximum difference between nodal radial basis functions
stays below 10\textsuperscript{-5}, for width parameters spanning 50
nodes. Under the same conditions, Figure \ref{fig:5}-d to e shows the maximum
change between consecutive radial basis functions is much larger than
10\textsuperscript{-2}. While not shown on the figure, this is trends is
also valid for basis functions centered on the domain boundaries. The
shape of the nodal radial basis functions there differs from the shape
of the functions in the domain interior, as shown in Figure \ref{fig:6}. However,
their construction is identical to the other functions and does not
require special treatment. The NRBF method presented here is not bound to Wendland functions. However, the decays of the coefficients is crucial suppressing the Gibbs phenomenon seen on Fig. \ref{fig:2} and they should be studied carefully before using other types of radial basis functions \cite{fornberg2008locality}, especially if they are infinitely smooth. 
\begin{figure}[ht]
\includegraphics[width=4in]{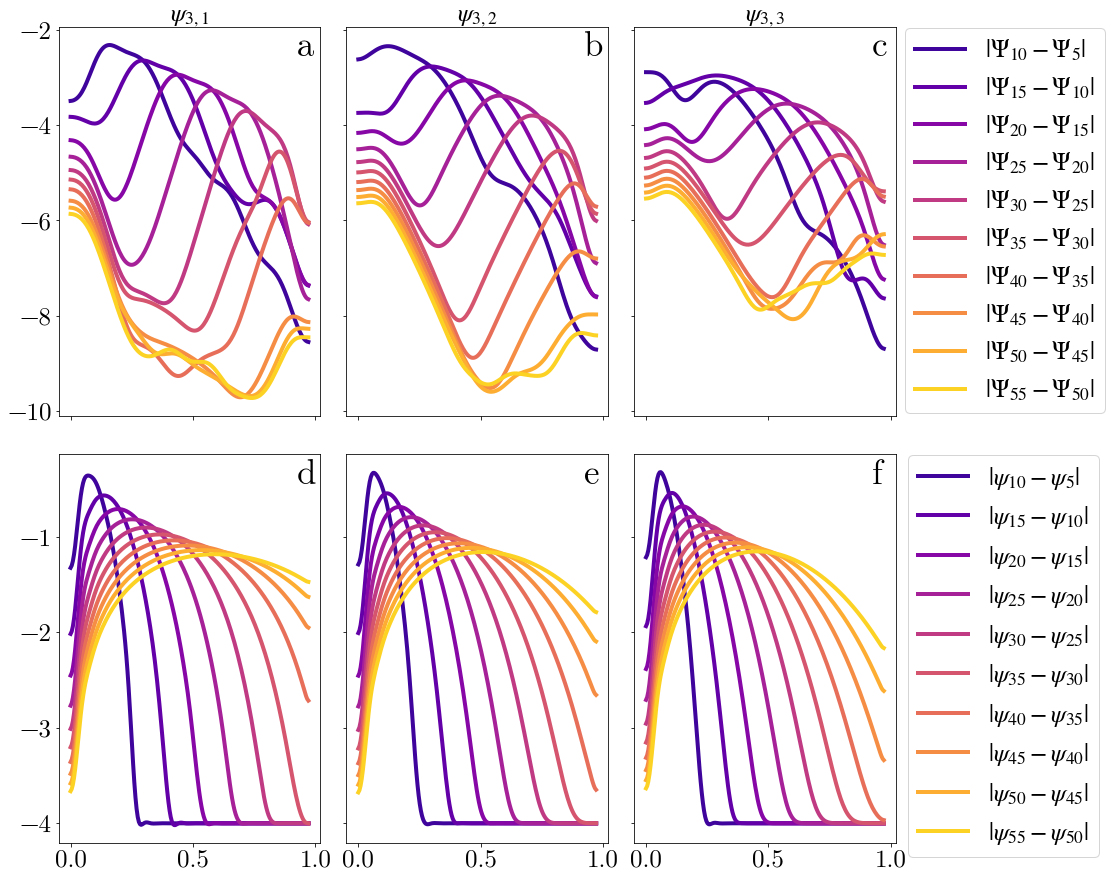}
\caption{The log\textsubscript{10} difference between nodal radial
basis functions using a) $\psi_{3,1}$, b) $\psi_{3,2}$
and c) $\psi_{3,3}$ with consecutive width parameters (given in
number of nodes). The difference was computed halfway between nodes
since difference at the nodes is 0 by construction. The
log\textsubscript{10} difference between two consecutive radial basis
functions c) $\psi_{3,1}$, b) $\psi_{3,2}$ and d)
$\psi_{3,3}$, using 10\textsuperscript{-4} cut-off.}
\label{fig:5}
\end{figure}
\begin{figure}[ht]
\includegraphics[width=4in]{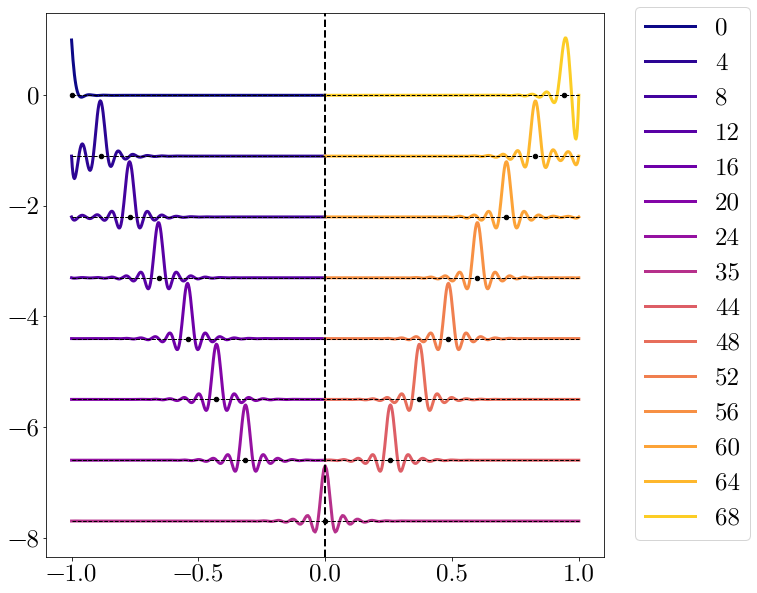}
\caption{Nodal radial basis functions for different support nodes,
indicated by a black dot). The total number of nodes is 71. The effect
of the boundary, while dramatic, does not really depend on the width
parameters for parameters larger than 50\%. A bias was added to each
function to improve the clarity of the plot.}
\label{fig:6}
\end{figure}

\section{The linear advection equation solved using Euler's backward
method}

Now that we have explored the basis properties of nodal radial basis functions, we would like to solve the following partial differential equation in multiple dimensions
\begin{equation}\label{eq:16}
\frac{\partial\rho}{\partial t} + \overrightarrow{\nabla}.(\rho\overrightarrow{u}) = S
\end{equation}
on the discretized domain \emph{U}. This equation is hyperbolic and
describes mass conservation, where $\rho$ is a mass density,
$\overrightarrow{u}$ is the velocity and \emph{S} is the source term.
The backward Euler time discretization scheme is given by
$$\rho_{t} + \overrightarrow{\nabla}.\left( \rho_{t}\overrightarrow{u} \right)\Delta t = \rho_{t - \Delta t} + S_{t}\Delta t.$$
It is convenient to drop the subscript term $t$ and rewrite the
equation as
\begin{equation}\label{eq:17}
\rho + \overrightarrow{\nabla}.\left( \rho\overrightarrow{u} \right)\Delta t = \rho_{- \Delta t} + S\Delta t = G
\end{equation}
where $\rho_{- \Delta t}$ is the solution at the previous time step.
\subsection{Nodal radial basis function
solver}
Using the nodal radial basis functions defined earlier, we can
interpolate the mass density $\rho$, the mass density flux
$\rho\overrightarrow{u}$ in Eq. (\ref{eq:17}) as
\begin{equation*}
\rho = \sum_{j}^{}\rho_{j}\Psi_{\alpha,x_{j}},
\end{equation*}
and
\begin{equation}\label{eq:19}
\rho\overrightarrow{u} = \sum_{j}^{}\rho_{j}{\overrightarrow{u}}_{j}\Psi_{\alpha,x_{j}},
\end{equation}
as well as
\begin{equation*}
G = \sum_{j}^{}G_{j}\Psi_{\alpha,x_{j}}.
\end{equation*}
Notice that we make no assumptions regarding the velocity distribution.
It could be space and time dependent at this point. Using NRBFs in Eq.
(\ref{eq:17}), we get
\begin{equation}\label{eq:21}
\sum_{j}^{}\rho_{j}\Psi_{\alpha,x_{j}} + \sum_{j}^{}{\sum_{k}^{}\rho_{j}}u_{k,j}\partial_{k}\Psi_{\alpha,x_{j}}\Delta t = \sum_{j}^{}G_{j}\Psi_{\alpha,x_{j}},
\end{equation}
where $u_{k,j}$ is the $k^{\text{th}}$ component of the vector
$\overrightarrow{u}$ and $\partial_{k}$ is the partial derivative
along the direction \emph{k}. So, we get
$\forall x_{i} \in U\,\sum_{j}^{}\rho_{j}\Psi_{\alpha,x_{j}}(x_{i}) + \sum_{j}^{}\rho_{j}\sum_{k}^{}u_{k,j}\partial_{k}\Psi_{\alpha,x_{j}}(x_{i})\Delta t = \sum_{j}^{}G_{j}\Psi_{\alpha,x_{j}}(x_{i})$.
Using Eq. (\ref{eq:6}) we now have
\begin{equation}\label{eq:22}
\forall x_{i} \in U\,\rho_{i} + \sum_{j}^{}\rho_{j}\sum_{k}^{}u_{k,j}\partial_{k}\Psi_{\alpha,x_{j}}(x_{i})\Delta t = G_{i}.
\end{equation}
Eq. (\ref{eq:22}) is valid for any number of spatial dimensions, is completely
agnostic of the node distribution and can be written in matrix form
\[(I - \Delta tA)\lbrack\rho\rbrack = \lbrack G\rbrack.\]
Here $I$ is the identity matrix and the elements of \emph{A} are
\begin{equation*}
A_{\text{ij}} = - \sum_{k}^{}u_{k,j}\partial_{k}\Psi_{\alpha,x_{j}}(x_{i}).
\end{equation*}
The solution to this system of equations is found by inverting the
matrix $(I - \Delta tA)$.
\begin{equation}\label{eq:24}
\lbrack\rho\rbrack = (I - \Delta tA)^{- 1}\lbrack G\rbrack.
\end{equation}
\begin{thrm} \emph{We can approximate the solution $[\rho]$ with $[\rho^*]$ using only matrix-vector products}
\begin{equation}\label{eq:27}
\lbrack\rho^{*}\rbrack = \sum_{k = 0}^{N}A\lbrack G_{k}\rbrack\Delta t^{k},
\end{equation}
\emph{where}
\begin{equation*}
\forall k > 0,\lbrack G_{k}\rbrack = A\lbrack G_{k - 1}\rbrack\,\emph{and}\, A\lbrack G_{0}\rbrack = \lbrack G\rbrack.
\end{equation*}
\end{thrm}
\begin{proof}
Since $(1 - x)^{- 1} = \sum_{k = 0}^{\infty}x^{k}$,
$(I - \Delta tA)^{- 1}$ can be approximated by a truncated series for
$\Delta t$ small enough, and we get
\begin{equation}\label{eq:25}
(I - \Delta t A)^{- 1} = \sum_{k = 0}^{N}{\Delta t^{k}A^{k}},
\end{equation}
where $A^{0} = I$. As a result, one can find the approximate solution
$\rho^{*}$ directly using
\begin{equation}\label{eq:26}
\lbrack\rho^{*}\rbrack = \sum_{k = 0}^{N}A^{k}\lbrack G\rbrack\Delta t^{k}.
\end{equation}
We also use the product of the matrix \emph{A} by a vector \emph{v} as $A(Av)$ rather than $(AA)v$ since it is more efficient to compute when the number of nodes is large.
\end{proof}

However, taking time steps small enough to warrant the approximation leading to Eq. (\ref{eq:26}) is not realistic in practice. However, if the velocity $\overrightarrow{u}$ and the source term \emph{S} vary on a time scale $\Delta T$, which is large compared to our choice of $\Delta t$, then the algorithm can ``step over'' the slow temporal change in source and density. Using Eq. (\ref{eq:24}) as our induction relation, we can start from a given solution at $t - \Delta T$ where \emph{P} is defined as $\Delta T = P\Delta t$. At this point, we can substantially reduce $\Delta t$ while increasing $P$ in such a way that the computational time step $\Delta T$ keeps the method stable. This method is implicit, and time stepping is not limited by the Courant--Friedrichs--Lewy (CLF) condition \cite{cour}. However, this method is still limited by a Nyquist-Shannon condition, as taking a time step $\Delta T$ much larger than the physical time evolution of the solution cannot capture the actual evolution of the solution. 
\begin{thrm}
We can approximate the solution at \emph{t} using a large time step $\Delta T=P\Delta t$
\begin{equation}\label{eq:29}
\lbrack\rho\rbrack = (I - \Delta tA)^{- P}\lbrack G_{- P\Delta t}\rbrack
\end{equation}
where
$\left\lbrack G_{- P\Delta t} \right\rbrack = \left\lbrack \rho_{- P\Delta t} \right\rbrack + \sum_{k = 0}^{P - 1}(I - \Delta tA)^{k}\left\lbrack S_{- P\Delta t} \right\rbrack\Delta t$.

\end{thrm}
\begin{proof}
Starting from Eq. (\ref{eq:24}) we get $$\lbrack\rho\rbrack = (I - \Delta tA)^{- 1}\left[ \rho_{- \Delta t}\right]+ (I - \Delta tA)^{- 1}\left[ S\right]\Delta t.$$ Since we suppose $S$ constant during this interval of time we now have $$\lbrack\rho\rbrack = (I - \Delta tA)^{- 2}\left(\left[ \rho_{- 2\Delta t}\right]+ \left[ S\right]\Delta t+ (I - \Delta tA)\left[ S\right]\Delta t\right).$$We find Eq. (\ref{eq:29}) by induction.
\end{proof}
Using
$\sum_{k = 0}^{P - 1}(1 - x)^{k} = \sum_{k = 0}^{P - 1}( - 1)^{k}\binom{P}{k + 1} x^{k}$
we get
\begin{equation*}
\left\lbrack G_{- \Delta T} \right\rbrack = \left\lbrack \rho_{- \Delta T} \right\rbrack - \sum_{k = 0}^{P - 1}{\left( - \frac{1}{P} \right)^{k + 1} \binom{P}{k + 1} A^{k}\left\lbrack S_{- \Delta T} \right\rbrack}\Delta T^{k + 1}.
\end{equation*}
Besides $A^{k}$ and $\Delta T^{k + 1}$,
$\left( \frac{1}{P} \right)^{k + 1} \binom{P}{k + 1} $
also decays quickly with \emph{k} and we can truncate the finite series
given above to its \emph{M} first terms as
\begin{equation*}
\left\lbrack G_{- \Delta T}^* \right\rbrack = \left\lbrack \rho_{- \Delta T}^* \right\rbrack - \sum_{k = 0}^{M}{\left( - \frac{1}{P} \right)^{k + 1}\binom{P}{k + 1}A^{k}\left\lbrack S_{- \Delta T} \right\rbrack\Delta T^{k + 1}}.
\end{equation*}
In the problem discussed in this paper, \emph{M} is between 10 and 20,
while \emph{P} can be as large as 10\textsuperscript{10}. We can recast
this series as a successive product of a matrix with a vector when
necessary, as we did in Eq. (\ref{eq:27}). So far, we kept the source term $S$
for completion.  We drop this term in the rest of the paper to simplify the discussion as it does not impact mathematical foundations of the method . Since
$(1 - x)^{- P} = \sum_{k = 0}^{\infty} \binom{k + P - 1}{P - 1}x^{k}$,
we can also write Eq. (\ref{eq:29}) as a truncated series where we keep the first
\emph{N} terms
\begin{equation}\label{eq:32}
\left\lbrack \rho^{*} \right\rbrack = \sum_{k = 0}^{N}{\left( \frac{1}{P} \right)^{k}\binom{k + P - 1}{P - 1}A^{k}\left\lbrack G_{- \Delta T}^* \right\rbrack\Delta T^{k}}.
\end{equation}
As for \emph{M}, \emph{N} can also be chosen between 10 and 20. Eq. (\ref{eq:32})
can be transformed into a successive product of a matrix by a vector, as
was done in Eq. (\ref{eq:27}). It is important to note that we only need to know
$\partial_{k}\Psi_{\alpha,x_{j}}(x_{i})$ to solve this is problem. So,
there is no need to compute the matrix inverse of
$\lbrack\Phi_{\alpha}\rbrack$ to get the coefficients
$\Omega_{\text{ij}}$. Eq. \ref{eq:7} shows that we can compute $\partial_{k}\Psi_{\alpha,x_{j}}(x_{i})$ by solving the linear system
\begin{equation*}
\begin{matrix}
\lbrack\Phi_{\alpha}\rbrack\lbrack\partial_{k}\Psi_{\alpha,x_{j}}\rbrack = \lbrack\partial_{k}\Phi_{\alpha}\rbrack \\
\end{matrix}
\end{equation*}
using Cholesky's factorization, which does not require any explicit
matrix inversion to solve Eq (33). So, the proposed solver uses neither
matrix inversions nor matrix products in its final form. It is only
based on matrix-vector multiplications, as other efficient solvers (e.g.
Ref. \cite{saad}).

\subsection{Radial basis function
solver}
We now focus on solving a similar problem using radial basis functions
and compared both methods. The mass density can be rewritten as
\begin{equation}\label{eq:34}
\rho = \sum_{j}^{}\omega_{\rho,j}\Phi_{\alpha,x_{j}}.
\end{equation}
If we suppose constant velocity across the whole domain, we have
\begin{equation}\label{eq:35}
\rho\overrightarrow{u} = \sum_{j}^{}\omega_{\rho,j}\Phi_{\alpha,x_{j}}\overrightarrow{u}.
\end{equation}
We will make this assumption in the remainder of this section. As we did
before, we take \emph{G} to be
\begin{equation*}
\def\labelenumi{(\arabic{enumi})}
\setcounter{enumi}{35}
G = \sum_{j}^{}\omega_{G,j}\Phi_{\alpha,x_{j}}.
\end{equation*}
So, we get $\forall x_{i} \in U\,$
\begin{equation}\label{eq:37}
\sum_{j}^{}\omega_{\rho,j}\Phi_{\alpha,x_{j}}(x_{i}) + \sum_{j}\omega_{\rho,j}\sum_{k}u_{k}\partial_{k}\Phi_{\alpha,x_{j}}(x_{i})\Delta t = \sum_{j}^{}\omega_{G,j}\Phi_{\alpha,x_{j}}(x_{i}).
\end{equation}
We can rewrite this system in matrix form as
\begin{equation}\label{eq:38}
\begin{matrix}
\lbrack\Phi_{\alpha} - \Delta tB\rbrack\lbrack\omega_{\rho}\rbrack = \lbrack\Phi_{\alpha}\rbrack\lbrack\omega_{G, - \Delta t}\rbrack \\
\end{matrix}
\end{equation}
where $B$ is the matrix derivative defined by
$$B_{\text{ij}} = - \sum_{k}^{}{u_{k}\partial_{k}\Phi_{\alpha,x_{j}}\left( x_{i} \right)}.$$

 If we want to use the larger time step $\Delta T=P\Delta t$, the
procedure described previously applied to equation Eq. (\ref{eq:38})
gives 
\begin{equation}\label{eq:39}
\lbrack\omega_{\rho}\rbrack = (\lbrack\Phi_{\alpha} - \Delta tB\rbrack^{- 1}\lbrack\Phi_{\alpha}\rbrack)^{P}\lbrack\omega_{G, - \Delta T}\rbrack
\end{equation}
However, the inverse of the matrix
$\lbrack\Phi_{\alpha} - \Delta tB\rbrack$ needs to be computed
explicitly here, even if we use a series approximation to compute the
matrix product
$(\lbrack\Phi_{\alpha} - \Delta tB\rbrack^{- 1}\left\lbrack \Phi_{\alpha} \right\rbrack)^{P}$.
 We can also rewrite Eq. (\ref{eq:38}) as 
\begin{equation*}
(I - \Delta tC)\lbrack\omega_{\rho}\rbrack = \lbrack\omega_{G}\rbrack
\end{equation*}
where $C = \lbrack\Phi_{\alpha}\rbrack^{- 1}B$. In this case we get,
\[\lbrack\omega_{\rho}\rbrack = (I - \Delta tC)^{- P}\left\lbrack \omega_{\rho, - \Delta T} \right\rbrack.\]
In this form, the system can be solved similarly to Eq. (\ref{eq:29}) and we can
use the same procedure to obtain an equation like Eq. (\ref{eq:32}). However, the
matrix inverse $\lbrack\Phi_{\alpha}\rbrack^{- 1}$ needs to be
computed explicitly, even when using a series approximation to compute  $(I - \Delta tC)^{-1}$.

\subsection{Comparison between the different
solvers}
Before embarking into a parameter scan to understand the limits of nodal
radial basis functions as an implicit advection equation solver, we
first would like to compare it to a solver using radial basis functions.
Unless otherwise indicated, the radial basis function-based solvers
(i.e. RBF and NRBF) used the Wendland function $\psi_{3,4}$, which
guarantees strictly positive definiteness in up to three dimensions. We
chose this function for its exceptional smoothness. Besides the RBF
solver, we also compared the NRBF solver to a standard implicit solver,
the second order centered implicit (CI) solver given by
\begin{equation*}
\rho_{i}^{n + 1} + \frac{u\Delta T}{12\Delta x}( - \rho_{i + 2}^{n + 1} + 8\rho_{i + 1}^{n + 1} - 8\rho_{i - 1}^{n + 1} + \rho_{i - 2}^{n + 1}) = \rho_{i}^{n},
\end{equation*}
This equation can be turned into a matrix form leading to a form similar
to Eq. (\ref{eq:29}). We also compared the method to the explicit Lax-Wendroff
(LW) discretization scheme given by
\begin{equation}\label{eq:43}
\rho_{i}^{n + 1} = \rho_{i}^{n} - \frac{\Delta T}{2\Delta x}u\left\lbrack \rho_{i + 1}^{n} - \rho_{i - 1}^{n} \right\rbrack + \frac{\Delta t^{2}}{2\Delta x^{2}}u^{2}\left\lbrack \rho_{i + 1}^{n} - 2\rho_{i}^{n} + \rho_{i - 1}^{n} \right\rbrack
\end{equation}
Both solver are notoriously non-conservative, allowing to check our conservative NRBF approach. To obtain an absolute measure of the error, we compared all numerical
solutions to a smooth solution of the advection equation Eq. (\ref{eq:16}) in one
dimension with no source, namely
\begin{equation*}
F_{\sigma}\left( x - x_{0} - ut \right) = 1 + {exp -}\left\lbrack \left( x - x_{0} - ut \right)/\sigma \right\rbrack^{2},
\end{equation*}
traveling from the left to the right. We used the value of the solution
$F_{\sigma}$ as a Dirichlet boundary condition for the
different methods, applied to three nodes to the left and right sides of
the domain. Our initial condition also used the solution
$F_{\sigma}$ with it peak set on the left boundary.

Since Eq. (\ref{eq:16}) is dimensionless, we took the velocity \emph{u} to be 1
and a domain {[}-2,2{]}. For this comparison, we discretized our domain
with 501 nodes. The peak of $F_{\sigma}$ starts at the left
boundary node and travels to the right. As written, the function peak
reaches the right boundary at t=4. The main reason for adding a constant
to the function is to let the solution to relax to 1 after the Gaussian
pulse traversed the domain, allowing to test the long-term stability of
the different solvers for non-trivial solutions. While other methods can
be used to solve this differential equation, we use here the same method
across all three implicit schemes to compare all the three implicit
solvers on the same footings.
\begin{figure}[ht]
\includegraphics[width=3.5in]{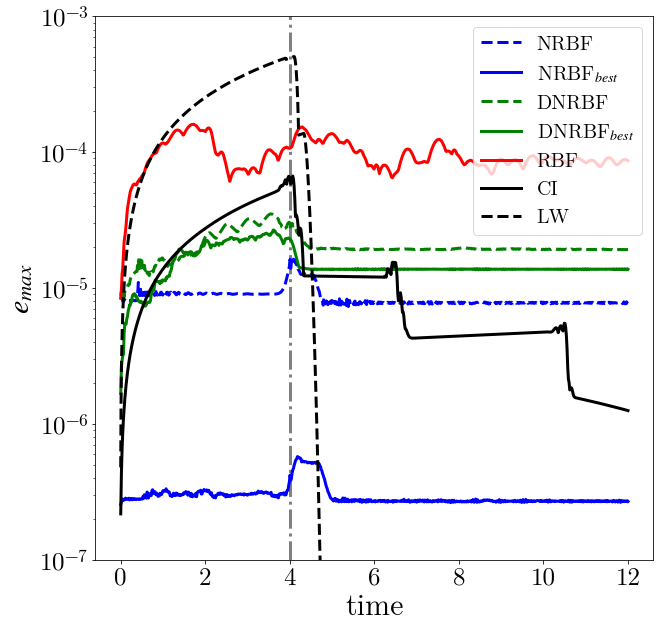}
\caption{Comparison between the error of three different implicit
methods using centered implicit finite differences (CI), radial basis
functions (RBF), nodal radial basis functions using a series
approximation for the matrix inversion (NRBF) or a direct solver
(DNRBF). The nodal NRBF methods using a large width parameter are
indicated using the subscript ``best''. The explicit Lax-Wendroff (LW)
method is shown for reference. The vertical dot-dashed gray line
corresponds to the time when the Gaussian peak reaches the right
boundary of the domain.}
\label{fig:7}
\end{figure}

Figure \ref{fig:7} shows the maximum error $e_{max}$ between the
true solution, given by Eq. (\ref{eq:43}), and the numerical solution, computed
using different solvers. As expected, the Lax-Wendroff solver (dashed
black line in Figure 7) comes last, with a linear increase of the error
(seen as a logarithmic curve in the log plot of Figure \ref{fig:7}). This solver
is known to have large numerical viscosity. It is also subject to the
CFL condition, forcing the time stepping to be much smaller than any
implicit methods, requiring 4 times as many as steps and degrading the
solution even further. The centered implicit solver (solid black line in
Figure \ref{fig:3}) does better than the explicit solver. But the error also
increases linearly (also seen as a logarithmic curve in the log plot of
Figure \ref{fig:7}) throughout the computation, a sign that both solvers are not
conservative.

The radial basis function solver (solid red in Figure \ref{fig:7}) starts with an
error similar to the Lax-Wendroff method but the error saturates rather
than increasing linearly. While numerical fluctuations are visible
throughout, they never turn unstable, keeping the error around
10\textsuperscript{-4}, even after the Gaussian has exited the domain.
This error was obtained with a width parameter 30 nodes (or 6\% the
total domain width). Larger parameters caused numerical instabilities.
The NRBF solver, using the same width parameter as the RBF solver
(dashed blue line in Figure \ref{fig:7}), performs much better, with an error
almost an order of magnitude smaller. Further, the error is virtually
constant throughout the computation, an indication that the method is
conservative, in the sense of Eq. (\ref{eq:13}). If it was not, the error would
keep increasing throughout the simulation. It is interesting to note
that the quality of the solution comes mostly from computing the inverse
of the matrix $(I-\Delta t A)^{-P}$ using a series approximation.
When the matrix inverse is computed directly (DNRBF, dashed green line
in Figure \ref{fig:7}), the error worsens noticeably. Note that both the DNRBF and
the NRBF methods give the exact same answer when using
$\psi_{3,3}$ instead of $\psi_{3,4}$,
indicating that the poor results of the direct solvers truly come from
roundoff errors in the inverse computed in Eq. (\ref{eq:39}). If we reduce the
width parameter enough to limit numerical instabilities inside the RBF
solver, then both solvers have the exact same error, independent of
the matrix inversion method.

While the maximum value of the width parameter is problem dependent for
both the nodal and standard radial basis function solvers, we cannot
avoid matrix inversions in the latter, as Eq. (\ref{eq:39}) shows. Since the
former solver uses a Cholesky decomposition, round-off errors are
negligible, allowing a width parameter twice as large. After that, the
NRBF method also becomes crippled by round-off errors. While we show
here the best case scenario, an increase in the width parameter does not
yield necessarily a better solution when the inverse of the matrix $(I-\Delta t A)^{-P}$ is computed directly
(DNRBF\textsubscript{best}, solid green line in Figure \ref{fig:7}). But there is
a clear improvement when the series approximation is used
(NRBF\textsubscript{best}, solid blue line in Figure \ref{fig:7}), with a
reduction of the error by more than an order of magnitude compared to
all other solvers.

\section{ Accuracy  of the implicit nodal radial basis function
solver}

Nodal radial basis functions can solve the linear advection equation
with greater precision compared to other standard methods. However,
these functions are defined implicitly, and it would be difficult to
determine the impact of the different parameters on the quality of the
solution. Yet, we have just seen that time stepping, the smoothness of
the radial basis functions and possibly the presence of a boundary can
affect the solution. But we are now in a position to use the `no
boundary' condition \cite{papa,grif} as an open
boundary condition at the right boundary. This condition was not used in
the previous section since the explicit finite difference scheme needed
boundary values at both boundaries and we wanted to treat all boundary
conditions on the same footing.

Now, the impact of time stepping is often an issue for computational
methods. While the method presented here is implicit and is not subject
to the CFL condition, it can only support a time stepping $\Delta T$ that
are three times the $\Delta T_{CFL}$ , the time step given by
the CFL, when using $\psi_{3,4}$ and four times the
$\Delta T_{CFL}$ when using $\psi_{3,3}$.
This limit can be increased further by using more points outside the
boundaries.

At present, it is not possible to define this limit precisely since the
nodal radial basis functions are not explicitly formed. However, we can
look in greater details at the sub-cycle timestep $\Delta t$. Figure \ref{fig:8}
shows how the sub-cycling (i.e the ratio $\Delta T$/$\Delta t$) impacts
the error for nodal radial basis functions formed by the smoothest
Wendland functions used in this paper, namely
$\psi_{3,4}$. When the sub-cycle ratio is small, the
solver is not conservative, and the error builds up linearly. In this
case, the average error of a simulation (solid line) is close to the
maximum error of that simulation (dashed line), which is also the end
error, while the minimum error (dotted line) is orders of magnitude
smaller than the average error. If we were to plot this error with time
it would behave like the error of the implicit centered method.
\begin{figure}[ht]
\includegraphics[width=3.4in]{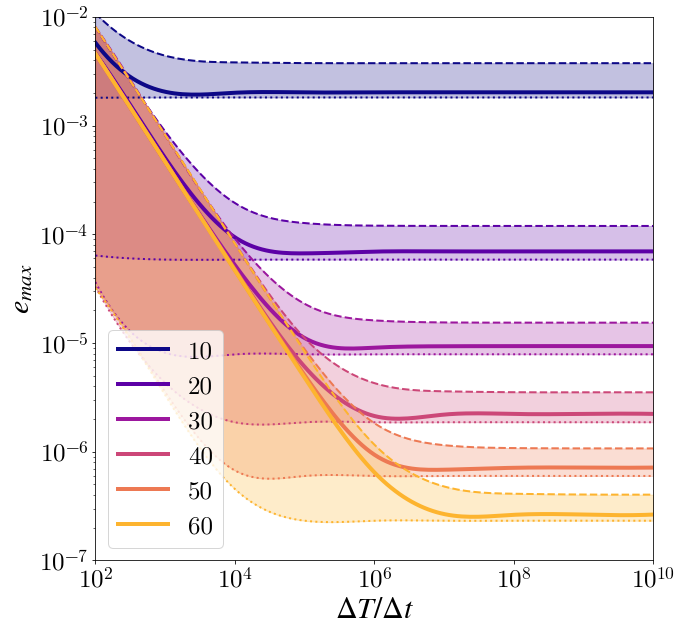}
\caption{Error for different ratios $\Delta T/\Delta t$, computed for different width
parameters, given in number of nodes. The solid lines give the average
error throughout the best NRBF simulations for different width
parameters, while the shaded region gives the error bracket, bounded by
the minimum (dotted line) and maximum (dashed line) errors. Typically,
when the average error is close to the minimum error, then the method is
conservative. In this case the maximum error comes from the bump caused
by the solution going through the boundary. When the average error is
close to the maximum error, then the solver is not conservative and lead
to a linear increase of the error, as seen in centered implicit or
Lax-Wendroff methods. The size of the width parameter is given as a
function of half the number of nodes spanned by the modal radial basis
function. The domain has a total of 501 nodes.}
\label{fig:8}
\end{figure}

While the error steadily diminishes with larger ratios, the method is
still not conservative until increasing the ratio does not yield better
error. Once the error has settled, we see that there is very little
difference between the minimum and maximum error, indicating that the
method is now conservative and the average error stays close to the
minimum error. When plotted against time in Figure \ref{fig:7}, the error would be
roughly flat throughout the simulation. For small width parameters, this
error quickly becomes insensitive to the ratio $\Delta T/\Delta t$ as it
gets dominated by interpolation error, controlled by the width
parameter. What is remarkable at this point is the fact that the method
is still conservative despite the large error. As the width parameter
becomes larger though, the kink in the error curve is pushed to higher
values of $\Delta T/\Delta t$ and the error diminishes steadily.

As we saw in Figure \ref{fig:3}, the width parameter and the smoothness greatly
change the way the function couples the neighboring nodes to the main
support node. Figure \ref{fig:9}-a shows that their impact is dramatic. When the
width parameter is small, the function smoothness does not impact the
error at all and the weak coupling between nodes (narrow stencil) yield
a relatively large error. As the width parameter becomes large and more
nodes are coupled, the error strongly depends on the function smoothness
and the width parameter. In the extreme case of
$\psi_{3,4}$, increasing the width parameter six times
reduces the error by four orders of magnitude. However, despite the
large error, the method remains conservative, showing a small error
variation throughout the simulation. We see that the error start to
saturate for even larger width parameters, since the shape of the nodal
radial basis functions becomes independent of the shape parameter. As
the coupling between nodes initially increases with large width
parameters, the number of nodes that need to be added outside of the
domain should also increase. This is necessary so boundary conditions
can be imposed with a spatial order that is consistent with the spatial
order of the method. Figure \ref{fig:9}-b shows that the error does improve with
more nodes located outside of the domain, but this change is weakly
dependent of on the number of boundary nodes.
\begin{figure}
a)\includegraphics[width=2in]{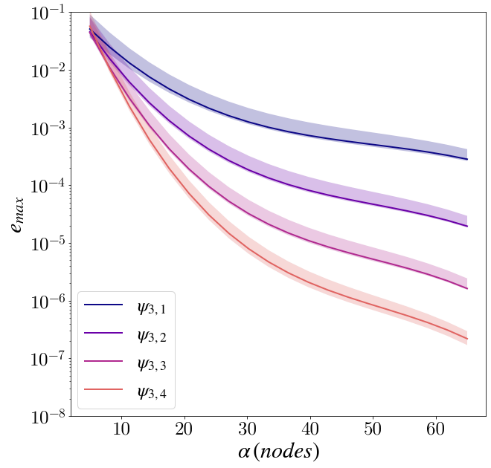}\\*
b)\includegraphics[width=2in]{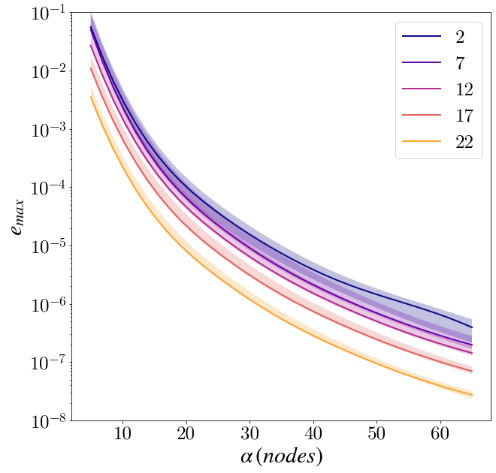}\\*
c)\includegraphics[width=2in]{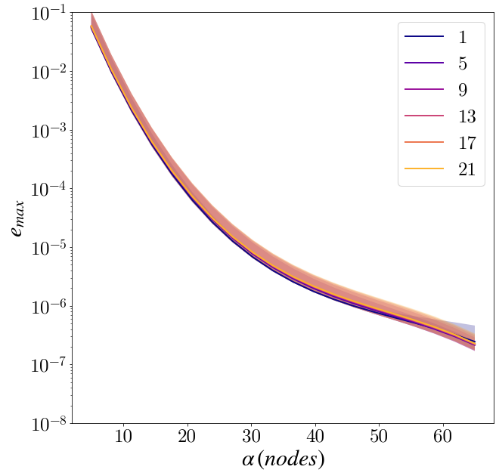}
\caption{Error for different width parameters (given as number of
nodes) for a) for the Wendland functions $\psi_{3,1}$,
$\psi_{3,2}$, $\psi_{3,3}$, and $\psi_{3,4}$, b)
a different number of boundary nodes and c) different Gaussian widths $\sigma$, defined in Eq. (\ref{eq:43}), both for the Wendland function
$\psi_{3,4}$. The solid lines give the average error throughout
the best NRBF simulations, while the shaded region gives the error
bracket. The domain has 501 nodes.}
\label{fig:9}
\end{figure}
Since we are looking at a particular solution of the partial
differential equation, which is relatively smooth, and the error is
mostly driven by the width parameter and the function smoothness, the
increase in the number of nodes does not provide a better approximation
of the solution, as shown in Figure \ref{fig:10}-a. However, numerical
instabilities start to become problematic if the width parameter is too
large, forcing the error to increase. Since we have an excellent
approximation with fewer nodes, we can generate a random grid
distribution with larger variations (up to +/-30\% from the homogeneous
node locations) to see the impact of the random distribution on the
overall quality of the solution. Here we let the simulation runs until
t=10 to verify that the solver reaches the steady state solution (1
according to Eq. (\ref{eq:43})) and stays stable throughout the simulation. While
the error increases noticeably, Figure \ref{fig:10}-b shows that the error is
bounded and remains reasonable even when nodes are displaced
substantially.
\begin{figure}[ht]
a)\includegraphics[width=2in]{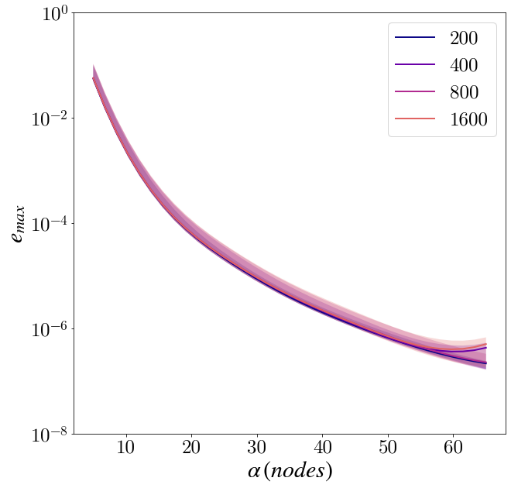}\\*
b)\includegraphics[width=2in]{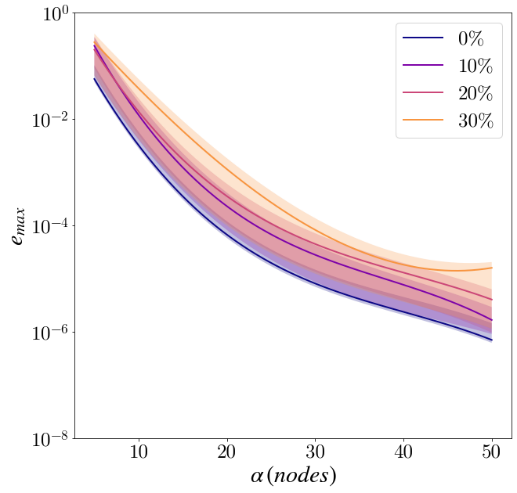}\\*
c)\includegraphics[width=2in]{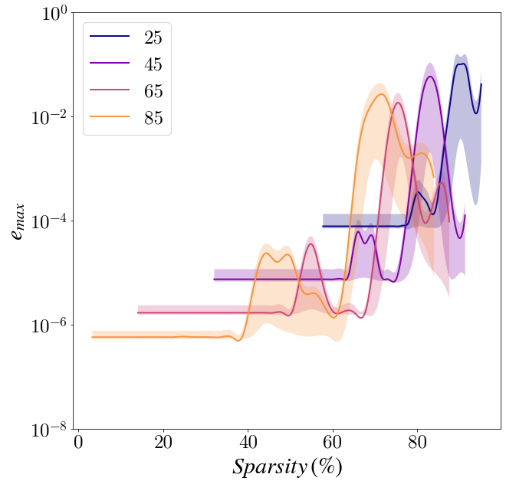}
\caption{ a) Using $\psi_{3,4}$ with 200, 400, 800 and 1600
nodes and b) with 200 nodes randomly distributed by 0\%, 10\%, 20\% and
30\% compared to the homogeneous internode distribution distance. c)
Error caused by truncation of the nodal radial basis function for
different width parameters (in number of nodes) for
$\psi_{3,3}$. The curve stops where the matrix is sparse without
truncation.}
\label{fig:10}
\end{figure}

As we discussed earlier, nodal radial basis functions can be truncated.
Figure \ref{fig:10}-c shows the nodal radial basis functions are not conservative
for extreme truncation. However, as we reduced sparsity, the method
becomes conservative and remains such until the minimum sparsity has
been reached. At this point, we simply stopped plotting the curve. Note
that the sparsity measurement is not absolute. As we saw in Figure \ref{fig:10}-a,
the precision comes from the total number of nodes contained inside the
nodal radial basis function, which is controlled by the width parameter,
rather than the total number of points contained inside the domain. When
solving real problems, the total number of points inside the domain will
increase while the precision is still dictated by the width parameter
and the function smoothness. So, the sparsity will increase drastically
for a given precision.

\section{The advection equation with non-homogeneous velocity.}

The main advantage of nodal radial basis functions over radial basis
functions is their ability to solve Eq. (\ref{eq:19}) with a variable velocity.
Limiting the discussion to the one-dimensional case, we can easily show
that the proposed implicit solver radial basis functions cannot solve
the advection equation when the velocity varies across the domain. In
this case, Eq. (\ref{eq:35}) becomes
\begin{equation}\label{eq:44}
\rho u = \sum_{j}^{}\omega_{\rho u,j}\Phi_{\alpha,x_{j}}.
\end{equation}
From this equation we turn Eq. (\ref{eq:35}) into
\begin{equation}\label{eq:45}
\forall x_{i} \in U\,\sum_{j}^{}\omega_{\rho,j}\Phi_{\alpha,x_{j}}(x_{i}) + \sum_{j}^{}\omega_{\rho u,j}\partial_{x}\Phi_{\alpha,x_{j}}(x_{i})\Delta t = \sum_{j}^{}\omega_{G,j}\Phi_{\alpha,x_{j}}(x_{i}).
\end{equation}
Unlike Eq. (\ref{eq:37}), Eq. (\ref{eq:45}) cannot be written as a linear system since
$\omega_{\rho u,j} \neq \omega_{\rho,j}u_{j}$. If we suppose the
contrary true, then we could write Eq. (\ref{eq:44}) as
\begin{equation*}
\rho u = \sum_{j}^{}{\omega_{\rho,j}u_{j}}\Phi_{\alpha,x_{j}}.
\end{equation*}

Using Eq. (\ref{eq:34}) we would get
\begin{equation*}
u\sum_{j}^{}\omega_{\rho,j}\Phi_{\alpha,x_{j}} = \sum_{j}^{}{\omega_{\rho,j}u_{j}}\Phi_{\alpha,x_{j}}.
\end{equation*}
So
\begin{equation*}
\forall x_{i} \in U,\, u\left( x_{i} \right)\sum_{j}^{}\omega_{\rho,j}\Phi_{\alpha,x_{j}}\left( x_{i} \right) = \sum_{j}^{}{\omega_{\rho,j}u_{j}}\Phi_{\alpha,x_{j}}\left( x_{i} \right),
\end{equation*}
which would give
\begin{equation*}
\forall x_{i} \in U,\,\sum_{j}\left( \omega_{\rho,j}u_{i} - \omega_{\rho,j}u_{j} \right)\Phi_{\alpha,x_{j}}\left( x_{i} \right) = 0.
\end{equation*}
Since ${\{\Phi}_{\alpha,x_{j}}\}$ are basis functions, we would get
\begin{equation*}
\forall i,j\ \,\omega_{\rho,j}u_{i} - \omega_{\rho,j}u_{j} = 0.
\end{equation*}

So, if $\omega_{\rho u,j} = \omega_{\rho,j}u_{j}$ then \emph{u} is
constant or $\rho = 0$ across the whole domain. Hence, the radial
basis function decomposition would need to use a non-linear implicit
method to solve this equation. The nodal radial basis function method
can solve this system directly, as it did in the constant velocity case.

As an example, we can solve Eq. (\ref{eq:16}) with the following velocity
distribution
\begin{equation*}
u_{\gamma,\sigma}\left( x - x_{c} \right) = 1 - \gamma{exp -}\left\lbrack \left( x - x_{c} \right)/\sigma \right\rbrack^{2},
\end{equation*}
where we took \emph{x\textsubscript{c}} to be the domain center node, $\gamma$
the damping gain and $\sigma$ the velocity distribution width. We chose $\sigma$ small
enough so the velocity at both boundaries is 1 and the discretization
can sample $u_{\gamma,\sigma}$ relatively well. We used here the
domain {[}-4,4{]} to guarantee that the velocity distribution is
virtually 1 at both boundaries. The velocity distribution used here is
shown in Figure \ref{fig:11}-a.
\begin{figure}[ht]
a)\includegraphics[width=2.2in]{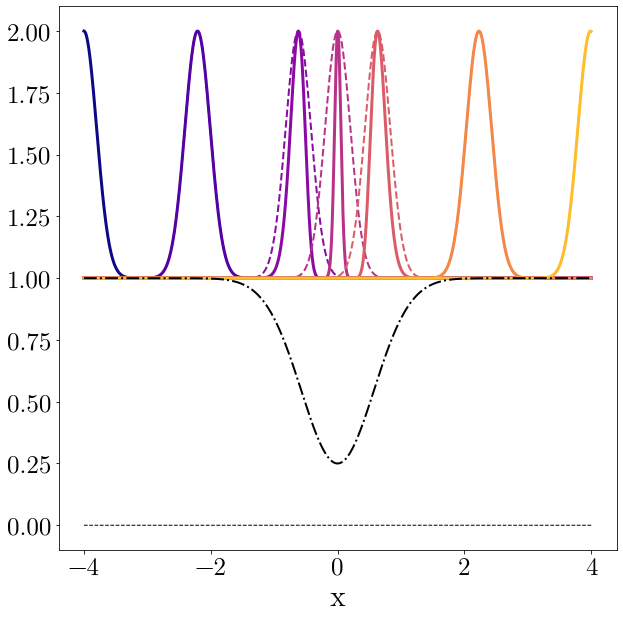}
b)\includegraphics[width=2.3in]{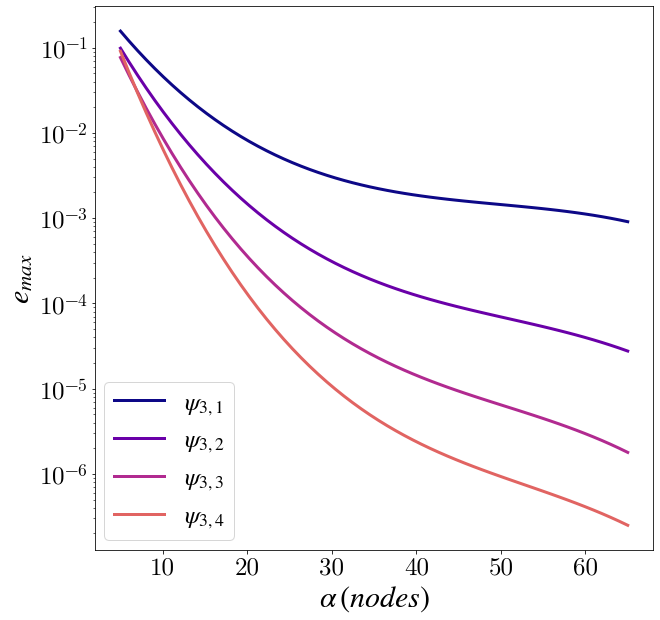}
\caption{a) The solution (solid line) to the differential equation as
a function of time from boundary to boundary, together with the Gaussian
pulse following a ballistic trajectory (dashed line) and the velocity
distribution across the domain (dot-dashed line). b) The maximum error
between the ballistic solution and the numerical solution at the right boundary for the Wendland functions $\psi_{3,1}$,
$\psi_{3,2}$, $\psi_{3,3}$, and $\psi_{3,4}$
with different width parameters (given in number of nodes). The domain has 501 nodes. The left boundary has Dirichlet's conditions. The right boundary has a `no boundary' condition.}
\label{fig:11}
\end{figure}
The solution of Eq. (\ref{eq:16}) is compressed in regions where
$u_{\gamma,\sigma}$ decreases and stretched in regions where
$u_{\gamma,\sigma}$ increases. The Gaussian pulse of Eq. (\ref{eq:43}) was used
as initial condition and as a Dirichlet condition of the left boundary
throughout the simulation. The right boundary here is open (i.e. the `no
boundary' boundary condition). Figure \ref{fig:11}-a shows the numerical solution
as it moves from the left to the right.

If our Gaussian pulse was to pass through the domain with a ballistic
trajectory, the location $x_t$ of the peak would be given
by
\begin{equation*}
x(t) - x\left( t_{0} \right) = \int_{t_{0}}^{t}{u_{\gamma,\sigma}\left( x(t') - x_{c} \right)\text{dt}}'.
\end{equation*}
Because the velocity distribution is symmetric with respect to
$x_c$, the Gaussian pulse following the ballistic
trajectory is also a solution of Eq. (\ref{eq:16}), but only in regions where
$u_{\gamma,\sigma} = 1$. The ballistic pulse traversing the domain is
also shown in Figure \ref{fig:11}-a. As we can see, it is indistinguishable from the numerical solution at the periphery of the domain.

We compared the numerical solution against the ballistic pulse at
$t_{final}$, when both peaks are located at the right
boundary. Figure \ref{fig:11}-b shows that the maximum error is comparable to the
case with constant velocity, and shown in Figure \ref{fig:9}-a.  The other errors were also similar and were not included to the paper. There is no
minimum and maximum error bracket here since the looked at the final
simulation time \emph{t\textsubscript{final}} and not at the whole time
series since we have not computed the error near
\emph{x\textsubscript{c}} with this method.

\section{Conclusion}
This paper shows how radial basis functions can be used to construct
implicitly a family of nodal radial basis functions on a discrete set
\emph{U} of nodes, which are interpolant of the translated impulse
function $\delta(x - x_{j})$. Unlike radial basis functions, which are
translate and scaled version of a single modal function, the nodal
radial basis functions depend on the node distribution. These functions
form an orthonormal basis on $\overline{U}$, the space of interpolant
operating on \emph{U}, leading to a simplified expression of the solver
obtained when discretizing the linear advection equation. This solver
can be extended trivially to the case where the velocity varies across
the whole domain.

One advantage of the nodal radial basis function method over the radial
basis function method is easily imposing boundary conditions. In
general, boundary conditions are given in term of the solution
(Dirichlet) or its derivatives (Neumann). Yet, the radial basis function
solver computes the solution in term of weights rather than the actual
solution values, using Eq. (\ref{eq:37}). So, at every time step, the solution
has to be computed at the domain nodes using Eq. (\ref{eq:34}), the boundary
conditions have to be applied, and then, the new weights have to be
computed using Eq. (\ref{eq:3}).

One possible issue with nodal radial basis functions comes from its
computational complexity. The Cholesky decomposition is
O(N\textsuperscript{3}), followed by N computations of the nodal radial
basis function derivatives, each O(N\textsuperscript{2}), to form the
matrix \emph{A}. So, we face another computation complexity that is
O(N\textsuperscript{3}), which is not present when using radial basis
functions. Each time step is O(N\textsuperscript{2}) after that. When
the time evolution of the equation requires a number of time steps that
is larger than N, the nodal radial basis function becomes more
advantageous. Indeed, we would need to go back and forth between the
actual value of the solution and its weight to impose boundary
conditions using radial basis functions, a transformation requiring
O(N\textsuperscript{2}) operations.

The O(N\textsuperscript{3}) dependence is problematic compared to the
centered implicit and Lax-Wendroff methods. But there are also simple
remedies to this ailment. Global nodal radial basis functions, as the
one used in this paper, can be truncated easily, leading to a
computational complexity that is O(N) where one node is connected to a
limited set of neighboring nodes, as opposed to all the nodes in the set
leading to an O(N\textsuperscript{2}) dependance. We could have also
used a partition of unity approach \cite{babu}, similar to the one used
with radial basis functions \cite{yoko}, also leading to an O(N) scaling.

While the present work mostly used global nodal radial
basis functions, it showed that truncating global nodal radial basis
functions can lead to sparse matrix algebra. It can be extended
relatively easily to partition-of-unity
methods \cite{grieb,yoko10}, to reduce the computational complexity from
O(N\textsuperscript{2}) to a O(N). This work was done in relatively
ideal conditions, staying away from solutions with sharp gradients
naturally arising from hyperbolic PDEs, which we plan to investigate
further using adaptive techniques. While adaptive meshing is not easy to implement, it clearly does not conflict with the methods presented herein and will be explored in future works.

\section*{Acknowledgements}

This research was supported in part by the NSF Awards PHY-1725178 and PHY-1943939.
\bibliographystyle{acm}
\bibliography{biblio}
\end{document}